\documentclass[12pt]{amsart}
\usepackage{ifxetex,ifluatex}
\newif\ifxetexorluatex
\ifxetex
  \xetexorluatextrue
\else
  \ifluatex
    \xetexorluatextrue
  \else
    \xetexorluatexfalse
  \fi
\fi

\ifxetexorluatex
  \usepackage{fontspec}
\else
  \usepackage[T1]{fontenc}
  \usepackage[utf8]{inputenc}
  \usepackage{lmodern}
  \fi
\usepackage{geometry}
\usepackage{float}
\usepackage{hyperref}
\usepackage{amsmath}
\usepackage{algorithmic}
\usepackage{amssymb, stmaryrd}
\usepackage{amsfonts}
\usepackage{mathrsfs}
\usepackage[capitalise]{cleveref}
\usepackage[titletoc]{appendix}

\newtheorem{lemma}{Lemma}
\newtheorem{proposition}{Proposition}
\def\VR{\kern-\arraycolsep\strut\vrule &\kern-\arraycolsep}
\def\vr{\kern-\arraycolsep & \kern-\arraycolsep}

\newcommand{\R}{\mathbb{R}}

\newcommand{\rD}{\mathrm{D}}
\newcommand{\rN}{\mathrm{N}}
\newcommand{\rgrad}{\mathsf{rgrad}}

\newcommand{\rhess}{\mathsf{rhess}}

\newcommand{\hatfY}{\hat{f}_{\bY}}
\newcommand{\hatfYY}{\hat{f}_{\bY\bY}}

\newcommand{\rK}{\mathrm{K}}

\newcommand{\sym}[1]{\mathrm{sym}_{#1}}
\newcommand{\asym}[1]{\mathrm{skew}_{#1}}
\newcommand{\KK}{\mathbb{K}}

\newcommand{\C}{\mathbb{C}}

\newcommand{\cE}{\mathcal{E}}

\newcommand{\St}[3]{\mathrm{St}_{#1, #2, #3}}

\newcommand{\Sd}[2]{\mathrm{S}^{+}_{#1, #2}}
\newcommand{\UU}[2]{\mathrm{U}_{#1, #2}}
\newcommand{\cH}{\mathcal{H}}

\newcommand{\ft}{\mathfrak{t}}

\newcommand{\cM}{\mathcal{M}}

\newcommand{\lb}{\llbracket}
\newcommand{\rb}{\rrbracket}

\newcommand{\bY}{Y}

\newcommand{\sfg}{\mathsf{g}}

\DeclareMathOperator{\diag}{diag}

\DeclareMathOperator{\xtrace}{xtrace}

\DeclareMathOperator{\JJ}{J}

\title{Riemannian gradient and Levi-Civita connection for fixed-rank matrices}
\author{Du Nguyen}
\email{\href{mailto:nguyendu@post.harvard.edu}{nguyendu@post.harvard.edu}}
\begin{document}

\begin{abstract}We provide formulas for Riemannian gradient and Levi-Civita connection for a family of metrics on fixed-rank matrix manifolds, based on nonconstant metrics on Stiefel manifolds.
\end{abstract}
\keywords{Optimization, Riemannian geometry, Stiefel manifold, Positive-definite, Positive-semidefinite, Levi-Civita connection, Hessian.}

\subjclass{65K10, 58C05, 49Q12, 53C25, 57Z20, 57Z25}
\maketitle
\section{Introduction}
Let $\KK$ be a field, either real ($\R$) or complex ($\C$). We apply the approach in \cite{NguyenRiemann} to compute Riemannian gradient and Hessian for the manifold of matrices in $\KK^{m\times n}$ with fixed-rank $p$, denoted by $\KK^{m\times n}_p$. Here, $m, n, p$ are positive integers. Let $\ft$ be the real transpose $T$ if $\KK$ is real and the hermitian transpose $H$ if $\KK$ is complex. A matrix $F\in \KK^{m\times n}_p$ factors to $F=UPV^{\ft}$ for $(U, P, V) \in \St{\KK}{p}{m}\times\Sd{\KK}{p}\times\St{\KK}{p}{n}$, where $\St{\KK}{p}{m}$ and $\St{\KK}{p}{n}$ are Stiefel manifolds (defined by $U^{\ft}U = I_p= V^{\ft}V$), and $\Sd{\KK}{p}$ is the manifold of positive definite matrices, (defined by $P^{\ft} = P$ and $P$ has positive eigenvalues). Such a factorization (we will call it PD-Stiefel factorization) always exists, via the SVD decomposition (with $P$ diagonal, hence symmetric). It is not unique. We call a matrix $O$ $\ft$-orthogonal if $O^{\ft}O = OO^{\ft} = I$. Let $\UU{\KK}{p}$ be the group of $\ft$-orthogonal matrices. If $(U, P, V)$ is a PD-Stiefel factorization of $F$, then $(UO^{\ft}, OPO^{\ft}, OV)$ is another factorization, and we have an equivalence relation $(U, P, V)\sim (UO^{\ft}, OPO^{\ft}, OV)$ with $O\in \UU{\KK}{p}$. Thus, $\KK^{m\times n}_p$ could be considered as a quotient $\St{\KK}{p}{m}\times\Sd{\KK}{p}\times\St{\KK}{p}{n}/\UU{\KK}{p}$. This approach was proposed in \cite{Mishra2014}. Here, we allow non constant ambient metrics on the two Stiefel manifolds $\St{\KK}{p}{m}$ and $\St{\KK}{p}{n}$, and we implement the complex case $\KK=\C$, which to our best knowledge has not been implemented before. The derivation follows from the framework proposed in \cite{NguyenRiemann}.
\section{Main results}
Let $\cE = \KK^{m\times p}\oplus \KK^{p\times p}\oplus \KK^{n\times p}$, what we call an {\it ambient space}. The tangent space of $\cM = \St{\KK}{p}{n}\times\Sd{\KK}{p}\times\St{\KK}{p}{m}$ at $(U, P, V)$ could be identified with a subspace of $\cE$, via the usual identification of tangent space of the Stiefel manifolds $\St{\KK}{p}{m}$ and $\St{\KK}{p}{n}$ as subspaces of $\KK^{m\times p}$ and $\KK^{n\times p}$ ($U^{\ft}\eta_U+\eta_U^{\ft}U = 0, V^{\ft}\eta_V+\eta_V^{\ft}V = 0$), while the tangent space of $\Sd{\KK}{p}$ is identified with $\ft$-symmetric matrices in $\KK^{p\times p}$, thus $\eta_P^{\ft} = \eta_P$. Here, a tangent vector of $\cM$ is represented by three components $(\eta_U, \eta_P, \eta_V)$. The action of an element $O\in \UU{\KK}{p}$ on an element $(U, P, V)$ by $(UO^{\ft}, OPO^{\ft}, OV)$ is free (no fixed point) and proper (as $\UU{\KK}{p}$ is compact), allowing us to identify $\KK^{m\times n}_p$ with the quotient manifold $\cM/\UU{\KK}{p}$. The recipe in \cite{NguyenRiemann} suggests we equip $\cM$ with an ambient metric $\sfg$, a self-adjoint operator-valued function from $\cM$ to $\cE$. If $(U, P, V)\in \cM$ and $\lb U, P, V\rb\in \cM/\UU{\KK}{p}$ is the corresponding equivalent class, this allows us to identify the tangent space of $\cM/\UU{\KK}{p}$ at $\lb U, P, V\rb$ with the horizontal space $\cH_Y$, the subspace of the tangent space $T_{(U, P, V)}\cM$ normal to the orbits of the group $\UU{\KK}{p}$. The horizontal space $\cH_Y$ could be identified as the nullspace of an operator $\JJ$ or image of an operator $\rN$. The paper gives us a number of explicit formulas that we can evaluate for the relevant Riemannian geometric quantities. Let $\hat{f}$ be a function defined on a neighborhood of $\cM$ in $\cE$, with gradient and hessian $\hatfY$, $\hatfYY$ respectively, the following describes the projection to the horizontal space, the horizontal Riemannian gradient, the Christoffel metric term, the Christoffel (Gamma) function, the Levi-Civita connection and the horizontal Riemannian Hessian operators (all evaluated at a manifold point $Y=(U, P, V)$):
\begin{equation}\label{eq:proj}
  \begin{gathered}
    \Pi_{\sfg} = \rN(\rN^{\ft}\sfg\rN)^{-1}\rN \\
    \Pi_{\sfg} = I - \sfg^{-1}\JJ(\JJ\sfg^{-1}\JJ^{\ft})^{-1}\JJ\\
  \rgrad_f = \Pi_{\sfg}\sfg^{-1}\hatfY
\end{gathered}
\end{equation}
\begin{equation}\label{eq:rhess}
  \begin{gathered}
  \rK(\xi, \eta) = \frac{1}{2}((\rD_{\xi}\sfg)\eta + (\rD_{\eta})\sfg\xi-\xtrace(\langle(\rD_\phi\sfg)\xi, \eta\rangle_{\cE}, \phi))\\
  \Gamma_{c}(\xi, \eta)=\Pi_{\sfg}\sfg^{-1}\rK(\xi, \eta)-(\rD_{\xi}\Pi_{\sfg})\eta\\
  \nabla_{\xi}\eta = \Pi_{\sfg}(\rD_{\xi}\imath \eta + \Gamma_c(\xi, \eta)\\ 
  \rhess^{11}_f\xi = \Pi_{\sfg}\sfg^{-1}(\hatfYY\xi + \sfg(\rD_{\xi}\Pi_{\sfg})(\sfg^{-1}\hatfY)
-(\rD_{\xi}\sfg)(\sfg^{-1}\hatfY)+\rK(\xi, \Pi_{\sfg}\sfg^{-1}\hatfY))
  \end{gathered}
\end{equation}
Here, $\imath\eta$ is the identification of the vector field $\eta$ to a $\cE$ valued-function, made explicit here to avoid confusion. $\xtrace$ is the index rising operator, which is simple for matrices:
\begin{equation}\begin{gathered}\xtrace(AbC, b) = A^{\ft}C^{\ft}\\
    \xtrace(Ab^{\ft}C, b) = CA
\end{gathered}    
\end{equation}
We will describe an one-to-one operator $\rN$ with range exactly the horizontal space and apply the above formulas for the gradient and Hessian (we will not describe the operator $\JJ$ in this paper although it can be constructed from the constraints described so far). This result extends \cite{Mishra2014}, which corresponds for the case all parameters $\alpha, \beta, \gamma$ are $1$. We recall the following from Proposition 7.2 in \cite{NguyenRiemann}

\begin{lemma}\label{lem:lyapunov}
Assume $\beta$ and $\delta$ are positive numbers. Consider the operator:
\begin{equation}\mathcal{L}(P)X = (\beta^{-1} -2\delta^{-1})X + \delta^{-1}(P^{-1}XP +PXP^{-1})
\end{equation}  
Let $P = C\Lambda C^{\ft}$ be a $\ft$-symmetric eigenvalue decomposition of $P$, where $\Lambda = \diag(\Lambda_1,\cdots,\Lambda_p)$ is the diagonal matrix of eigenvalues. Then
\begin{equation}\label{eq:ex_lyapunov}
  \begin{gathered}
    \mathcal{L}(P)^{-1}Z = C\{(C^{\ft}ZC) / M\} C^{\ft}
    \end{gathered}
\end{equation}
with $M \in \KK^{p\times p}$ is a matrix with entries $M_{ij} = \beta^{-1} - 2\delta^{-1} +\delta^{-1}(\Lambda_i^{-1}\Lambda_j + \Lambda_i\Lambda_j^{-1})$ and $/$ is a by-entry division. 
\end{lemma}
The proof is a straightforward substitution. $M_{ij} > 0$ by the AGM inequality.
\begin{proposition}
Let $[\omega_U, \omega_P, \omega_V] \in \cE = \KK^{m\times p}\oplus \KK^{p\times p}\oplus \KK^{n\times p}$, the following operator on $\cE$, which is positive-definite, gives $\cE$ an inner product that induces a Riemannian metric on $\cM$:
\begin{equation}\label{eq:metric}
  \sfg[\omega_U, \omega_P, \omega_V] = [\alpha_0\omega_U + (\alpha_1-\alpha_0)UU^{\ft}\omega_U, \beta P^{-1}\omega_P P^{-1}, \gamma_0\omega_V + (\gamma_1-\gamma_0)VV^{\ft}\omega_V]
\end{equation}  
where $\alpha_0, \alpha_1, \gamma_0,\gamma_1, \beta$ are positive numbers. Set $\delta = \alpha_1+\gamma_1$.
The projection in \cref{eq:proj} is given by:
\begin{equation}\label{eq:projfr}
  \begin{gathered}
    \Pi_{\sfg}(UPV^{\ft})[\omega_U, \omega_P, \omega_V] =
       [U\{-\gamma_1D^-+\delta^{-1}(P^{-1}D^+ -D^+P^{-1})\} + \omega_U-UU^{\ft}\omega_U,\\
         \beta^{-1} D^+,V\{\alpha_1D^-+ \delta^{-1}(P^{-1}D^+ -D^+P^{-1}) \} +\omega_V-VV^{\ft}\omega_V]\\
\text{ with }       D^- = \delta^{-1}\asym{\ft}(V^{\ft}\omega_V - U^{\ft}\omega_U)\\
D^+ =\mathcal{L}(P)^{-1}\sym{\ft}(\beta^{-1}\omega_P+\alpha_1\delta^{-1}(U^{\ft}\omega_UP -PU^{\ft}\omega_U) + \gamma_1\delta^{-1}(V^{\ft}\omega_VP - PV^{\ft}\omega_V))
\end{gathered}
\end{equation}
To compute the second order terms in \cref{eq:rhess}, the Christoffel metric term is
\begin{equation}\begin{gathered}
  \rK(\eta, \xi) = [
    (\alpha_1-\alpha_0)(U\sym{\ft}(\eta_U^{\ft}\xi_U)-(\eta_U\xi_U^{\ft} + \xi_U\eta_U^{\ft})U),\\
    -\beta\sym{\ft}(P^{-1}\eta_PP^{-1}\xi_PP^{-1}),\\
      (\gamma_1-\gamma_0)(V\sym{\ft}(\eta_V^{\ft}\xi_V)-(\eta_V\xi_V^{\ft} +
                           \xi_V\eta_V^{\ft})V)]
  \end{gathered}
\end{equation}
The directional derivatives $(\rD_{\xi}\Pi)\omega$ is computed by differentiating \cref{eq:projfr}, using:
\begin{equation}
  \begin{gathered}
  \rD_{\xi}D^- = \delta^{-1}\asym{\ft}(\xi_V^{\ft}\omega_V - \xi_U^{\ft}\omega_U)\\
\rD_{\xi}D^+  =\mathcal{L}(P)^{-1}\{ \rD_{\xi}\sym{\ft}(\alpha_1\delta^{-1}(U^{\ft}\omega_UP -PU^{\ft}\omega_U) + \gamma_1\delta^{-1}(V^{\ft}\omega_VP - PV^{\ft}\omega_V))\\
 - \delta^{-1}(\xi_P  D^+ P^{-1} + P^{-1} D^+ \xi_P -
                      P  D^+  P^{-1} \xi_P  P^{-1} -
                      P^{-1}  \xi_P  P^{-1}  D^+  P)\}
\end{gathered}\end{equation}

\end{proposition}

\begin{proof}
With the vertical vectors $Uq, qP - Pq, Vq$, the horizontal condition with the metric in \cref{eq:metric} is $\alpha_1 U^{\ft}\omega_U +\gamma_1V^{\ft}\omega_V + \beta \omega_PP^{-1} -\beta P^{-1}\omega_P=0$, here $[\omega_U, \omega_P, \omega_V]$ is an element of the ambient space $\cE = \KK^{m\times p}\oplus \KK^{p\times p}\oplus \KK^{n\times p}$. We will use the symbols $A^+=\sym{\ft}A, A^-=\asym{\ft}A$ for a matrix $A$. Map $[B, D, C] \in \cE_{\rN} = \KK^{(m-p)\times p}\oplus \KK^{p\times p}\oplus\KK^{(n-p)\times p}$ to $\cE$ by defining $\rN[B, D, C]) = [(\rN[B, D, C])_U, (\rN[B, D, C])_P, (\rN[B, D, C])_V]$ with:
$$\begin{gathered}
  (\rN[B, D, C])_U = U\{-\gamma_1D^-+(\alpha_1+\gamma_1)^{-1}(P^{-1}D^+ -D^+P^{-1})\} + U_0B\\
  (\rN[B, D, C])_P =\beta^{-1} D^+\\
  (\rN[B, D, C])_V =  V\{\alpha_1D^-+ (\alpha_1+\gamma_1)^{-1}(P^{-1}D^+ -D^+P^{-1}) \} +V_0 C
\end{gathered}$$
It is a linear map, satisfying the Stiefel tangent condition on the $U$ and $V$ components, the $P$ component is symmetric, and the horizontal condition is satisfied. It is also clearly a one-to-one map, and a dimensional count shows it is onto the tangent space. With $\delta = \alpha_1+\gamma_1$, we can compute directly:
$$\begin{gathered}\rN^{\ft}[\omega_U, \omega_P, \omega_V] = [(\rN^{\ft}[\omega_U, \omega_P, \omega_V])_B, (\rN^{\ft}[\omega_U, \omega_P, \omega_V])_D, (\rN^{\ft}[\omega_U, \omega_P, \omega_V])_C]\\
(\rN^{\ft}[\omega_U, \omega_P, \omega_V])_B = U_0^{\ft}\omega_U\\
  (\rN^{\ft}[\omega_U, \omega_P, \omega_V])_P =  \sym{\ft}(\beta^{-1}\omega_P+\delta^{-1} P^{-1}U^{\ft}\omega_U -\\
  \delta^{-1} U^{\ft}\omega_U P^{-1} +\delta^{-1} P^{-1}V^{\ft}\omega_V -\delta^{-1} V^{\ft}\omega_V P^{-1}) -\\
\gamma_1\asym{\ft}(U^{\ft}\omega_U)+\alpha_1\asym{\ft}(V^{\ft}\omega_V)\\
(\rN^{\ft}[\omega_U, \omega_P, \omega_V])_C =V_0^{\ft}\omega_V)\end{gathered}$$
$$\begin{gathered}\rN^{\ft}\sfg[\omega_U, \omega_P, \omega_V] = [\alpha_0U_0^{\ft}\omega_U,\\
    \sym{\ft}( P^{-1}\omega_PP^{-1}+\alpha_1\delta^{-1} P^{-1}U^{\ft}\omega_U -\alpha_1\delta^{-1} U^{\ft}\omega_U P^{-1} +\\
\gamma_1\delta^{-1} P^{-1}V^{\ft}\omega_V -\gamma_1\delta^{-1} V^{\ft}\omega_V P^{-1})-\alpha_1\gamma_1\asym{\ft}(U^{\ft}\omega_U )+\alpha_1\gamma_1\asym{\ft}(V^{\ft}\omega_V),\\
\gamma_0V_0^{\ft}\omega_V]\end{gathered}$$
Hence:
$$\rN^{\ft}\sfg \rN[B, D, C] = [\alpha_0B, \rN^{\ft}\sfg \rN[B, D, C]_D , \gamma_0C]$$
With $\rN^{\ft}\sfg \rN[B, D, C]_D = \rN^{\ft}\sfg \rN[B, D, C]_D^+ + \rN^{\ft}\sfg \rN[B, D, C]_D^{-}$ where:
$$\rN^{\ft}\sfg \rN[B, D, C]_D^+ = (\frac{1}{\beta} -\frac{2}{\alpha_1 + \gamma_1})P^{-1}D^+P^{-1} +\frac{1}{\alpha_1 +\gamma_1}(P^{-2}D^+ +D^+P^{-2}) $$
$$\rN^{\ft}\sfg \rN(B, D, C)_D^- =(\alpha_1^2\gamma_1+\alpha_1\gamma_1^2)D^-$$
Therefore, to solve $\rN^{\ft}\sfg \rN(B, D, C)= [\hat{B}, \hat{D}, \hat{C}]$, it is clear $B = \alpha_0^{-1}\hat{B}$, $C = \gamma_0^{-1}\hat{C}$ and $D^- = (\alpha_1\gamma_1\delta)^{-1}\hat{D}$.  Note $\mathcal{L}(P)D^+= P\rN^{\ft}\sfg \rN[B, D, C]_D^+P$, with $\mathcal{L}(P)$ as in \cref{lem:lyapunov}.

By the formula $\Pi_{\sfg} = \rN(\rN^{\ft}\sfg\rN)^{-1}\rN^{\ft}\sfg$, the projection of $[\omega_U, \omega_P, \omega_V]\in \cE$ could be evaluated by $\rN[B, D, C]$, where $[B, D, C]$ are solutions of $\rN^{\ft}\sfg\rN[B, C, D]  = \rN^{\ft}\sfg[\omega_U, \omega_P, \omega_V]$, therefore $D^+ = \mathcal{L}(P)^{-1}P(\rN^{\ft}\sfg[\omega_U, \omega_P, \omega_V])_P^+P$, $D^- = \delta^{-1}\asym{\ft}(V^{\ft}\omega_V - U^{\ft}\omega_U)$ and then:
\begin{equation}\label{eq:prj_fr}
  \begin{gathered}
    \Pi_{\sfg}(UPV^{\ft})[\omega_U, \omega_P, \omega_V] =\\
       [U\{-\gamma_1D^-+\frac{1}{\alpha_1+\gamma_1}(P^{-1}D^+ -D^+P^{-1})\} + \omega_U-UU^{\ft}\omega_U, \beta^{-1} D^+,\\
   V\{\alpha_1D^-+ \frac{1}{\alpha_1+\gamma_1}(P^{-1}D^+ -D^+P^{-1}) \} +\omega_V-VV^{\ft}\omega_V]\\
\end{gathered}
\end{equation}
The expression for $\rK$ is straightforward. Derive \cref{eq:prj_fr} by the usual matrix calculus rules, it is clear:
$$\rD_{\xi}D^- = \delta^{-1}\asym{\ft}(\xi_V^{\ft}\omega_V - \xi_U^{\ft}\omega_U)$$
For $\rD_{\xi}D^+$, we differentiate the defining equation for $D^+$:
$$\begin{gathered}
\mathcal{L}(P)\rD_{\xi}D^+ + \delta^{-1}(
\xi_P  D^+ P^{-1} + P^{-1} D^+ \xi_P -
                      P  D^+  P^{-1} \xi_P  P^{-1} -
                      P^{-1}  \xi_P  P^{-1}  D^+  P)  = \\
                      \rD_{\xi}\sym{\ft}(\alpha_1\delta^{-1}(U^{\ft}\omega_UP -PU^{\ft}\omega_U) + \gamma_1\delta^{-1}(V^{\ft}\omega_VP - PV^{\ft}\omega_V))
\end{gathered}$$
we can then solve for $\rD_{\xi}D^+$.
\end{proof}

\section{Geodesics and implementation} Note that $\cM$ is a complete manifold (as its components are complete, moreover their geodesics are known), a Frechet mean problem with the canonical metric on the Stiefel components ($\alpha_1 = \gamma_1 = \frac{1}{2}, \alpha_0=\gamma_0 = 1$) may be more tractable than other metrics for this manifold. See \cite{Zimmermann} for the log problem for Stiefel manifolds with canonical metric. The geodesics $F(t)$, represented by $U, P, V$, with $F(0) = (U, P, V)$ and $\dot{F}(0) = (\eta_U, \eta_P, \eta_V)$ is $\lb U(t), P^{1/2}\exp(t P^{-1/2}\eta_P P^{-1/2})P^{1/2}, V(t)\rb$, with $U(t)$ and $V(t)$ are Stiefel geodesics as in Proposition 5.1 of \cite{NguyenRiemann}.

We implemented the real and complex fixed-rank manifolds in \cite{Nguyen2020riemann}, with manifold classes RealFixedRank and ComplexFixedRank. We provided symbolic derivation of some formulas, together with a quadratic optimization problem as a numerical example.
  
\end{document}